\documentclass[leqno,12pt]{amsart}
\usepackage{amscd,amsfonts,amsmath,amssymb,amstext,amsthm}
\newtheorem{theorem}{Theorem}
\newtheorem{lemma}{Lemma}[section]
\newtheorem{proposition}{Proposition}[section]
\newtheorem{corollary}{Corollary}[section]
\newtheorem{definition}{\it Definiton}[section]

\numberwithin{equation}{section}
\begin{document}
\title{Crofton formulae for products}
\author{Dmitri Akhiezer and Boris Kazarnovskii}
\address {Institute for Information Transmission Problems \newline
19 B.Karetny per.,127051, Moscow, Russia,\newline
{\rm D.A.:} 
{\it akhiezer@iitp.ru},  
{\rm B.K.:} 
{\it kazbori@gmail.com}.}
\thanks{Research was carried out at the Institute for Information Transmission Problems of the Russian Academy of Sciences}

\begin{abstract} It is shown how new integral-geometric formulae can be obtained from the existing formulae of Crofton type. In particular, for classical Crofton formulae in which the answer depends on the Riemannian volume, 
we obtain generalizations in terms of the mixed Riemannian volume defined in the paper. The method is based on the calculations
in the ring of normal densities constructed in the previous work of the authors.
\end{abstract}
\renewcommand{\subjclassname}
{\textup{2010} Mathematics Subject Classification}
\maketitle
\renewcommand{\thefootnote}{}
\maketitle

\section{Introduction}\label{intro}

Let $Y$ be a manifold with a family of submanifolds $\{Y_\gamma\}_{\gamma\in\Gamma}$,
where
${\rm codim}\ Y_\gamma= n$ and the parametrizing manifold $\Gamma$ is equipped with a measure $\mu$.
In this setting, we will say that we have
{\it Crofton data}, to be denoted by ${\{Y, \Gamma, \mu}\}$.
Sometimes it is possible to compute, 
for any $n$-dimensional submanifold $M\subset Y$,
the average number    
$\mathfrak M_{Y, \Gamma, \mu}(M) $ of intersection points of $M$ and $Y_\gamma $.
Usually the answer is given by an $n$-density $\Omega_{Y, \Gamma, \mu}$ on $Y$,
such that 
\begin{equation}
\mathfrak M _{Y, \Gamma, \mu}(M)=\int_M\Omega _{Y, \Gamma, \mu}. \label{croftonformula}
\end{equation}
The density $\Omega _{Y, \Gamma, \mu}$ is called the {\it Crofton density}.
As a rule, Crofton data appear in the situation when the measure $\mu $
is constructed on the basis of some geometric structure on $Y$. 
If the Crofton density can be explicitly computed in terms of this geometric structure
then one calls 
(\ref{croftonformula}) the {\it Crofton formula.}
For example, in most known classical cases $Y$ is a Riemannian manifold 
and the restriction of $\Omega_{Y,\Gamma,\mu}$ to $M$ is, up to a factor independent of $M$, the volume density
of the induced metric, see \cite{Sa}.

Our main result
produces a Crofton formula for the product of Crofton data
$\{Y_1\times Y_2, \Gamma _1 \times \Gamma _2, \mu_1\times\mu_2\}$,
once such a formula is established for $\{Y_1,\Gamma_1,\mu_1\}$ and $\{Y_2, \Gamma _2,\mu_2\}$.
In the holomorphic case one has the Crofton formula for $Y={\mathbb P}^N$ with $Y_\gamma $ being 
projective subspaces of codimension $n$, see e.g.\,\cite{Sh},\cite{Shi}. In particular, for $n=1$ we get the Crofton formula
for curves in ${\mathbb P}^N$, from which one can deduce
a similar formula for surfaces in ${\mathbb P}^{N_1}\times{\mathbb P}^{N_2}$, see \cite{Ka}.
In these holomorphic versions of the Crofton formula,
we have some differential form in place of a Crofton density
and these forms are multiplied when we move to the product space.
Following this recipe in the real case, one has to multiply densities instead of exterior forms.
However, this is in general impossible, but we introduce a special class of {\it normal densities} which can be multiplied,
see Sect.\ref{preliminaries}.

Normal densities on a manifold $X$ of dimension $n$ are the analogue of, on the one hand, differential forms on $X$ and,
on the
other hand, affine Crofton densities on ${\mathbb R}^n$ defined by I.M.Gelfand an M.M.Smirnov in \cite{GS}.
It is shown in \cite{AK} that normal densities form a graded commutative ring ${\mathfrak n}(X) $ depending
contravariantly on $X$.
Under some natural assumptons, we will prove that
Crofton densities are normal and behave nicely in products.
This leads to an algorithm giving new integral-geometric formulae
on the basis of existing formulae of Crofton type. In particular, we deduce the Crofton formula
for the product of spheres.

We now state our main results.
Given
Crofton data  ${\{Y, \Gamma, \mu}\}$,  
put
$$X=\{(y,\gamma) \in Y \times \Gamma \ \vert \ y \in Y_\gamma\}$$
and denote by $p: X\to Y$ and $q: X\to \Gamma $ the projection maps. 
Suppose
that $p, q$ are locally trivial fiber bundles
and $\mu $ is given by a positive top-order smooth density $\Phi$, i.e., $\mu (B) = \int _B\Phi $ 
for any relatively compact open subset $B \subset \Gamma $.  Then, for any submanifold $M \subset Y$
of dimension $n$,
we have 
$$\int _{\gamma \in \Gamma} {\#}(M\cap Y_\gamma) \cdot d\mu =\int_M p_*q^*\Phi ,$$
where $p_*$ and $q^*$ stand for push-forward and pull-back operations for densities.  
To get finite expressions, we assume further
that the fiber of $p$ is compact, see \cite{PF} for the details.
\begin{theorem} \label{crofton density}The Crofton density $p_*q^*\Phi$, as well as its restriction to
any submanifold $M \subset Y$,  is normal.
\end{theorem}
\noindent
For the product of two Crofton data of the above type, the Crofton density exists and is equal
to the product of Crofton densities of the factors. More precisely, we have the following theorem.
 
\begin{theorem}\label{crofton product} 
Let $\{Y_1, \Gamma _1, \mu_1\}$ and $\{Y_2, \Gamma _2 ,\mu _2\}$ be two Crofton data and $\pi_1, \pi _2$
the projection maps of $Y _1 \times Y _2$ to $Y_1$ and $ Y_2$, respectively.
Then $\Omega _{Y_1\times Y_2, \Gamma _1 \times \Gamma _2, \mu_1 \times \mu_2}
=\pi_1^*\Omega _{Y_1, \Gamma _1, \mu _1}\cdot \pi _2^*\Omega _{Y_2, \Gamma _2, \mu_2}$,
the restrictions $\Omega _{M,i}$ of $\pi _i^*\Omega _{Y_i,\Gamma _i, \mu_i}$, $i =1,2$, to any $M \subset Y_1\times Y_2$
are normal and

$$\int _{(\gamma _1,\gamma _2 )\in {\Gamma _1 \times \Gamma _2}}{\#}(M\cap (Y_{\gamma _1} \times Y_{\gamma _2}) )
\cdot d\mu_1\cdot d \mu_2=
\int _M\Omega _{M,1}\cdot \Omega _{M,2}.$$

\end{theorem}

Theorem\,\ref{crofton density} is proved in Sect.\ref{decomp}. Theorem\,\ref{crofton product} is proved in Sect.\ref{prod}.

There are classical examples of Crofton data $\{Y, \Gamma, \mu\}$, such that $Y=(Y,g)$ is a Riemannian manifold with metric $g$ and
$\Omega _{Y,\Gamma,\mu} =C\cdot {\rm vol}_{n,g}$,
where the $n$-density ${\rm vol}_{n,g}$ assigns to $\xi_1 \wedge \ldots \wedge \xi_n \in \bigwedge ^nT_y(Y)$
the $n$-dimensional $g$-volume of the parallelotope generated by $\xi_i$. In this case,
${\mathfrak M}_{Y,\Gamma,\mu}(M) = C\cdot{\rm vol}_g(M)$,
where $C=C(n)$ is some constant.
Assume there is a number of such data $\{Y_i , \Gamma _i, \mu_i\}$, where $Y_i = (Y_i, g_i)$.
In Sect.\ref{mixedriemann} we define the mixed Riemannian volume ${\rm vol}_{g_1, \ldots, g_n}(M)$ under a mixing $n$-tuple of non-negative quadratic forms $g_1, \ldots, g_n$
on a given manifold $M$, ${\rm dim} M=n$. If these forms coincide with a metric $g$ on $M$ 
then the mixed volume equals ${\rm vol}_g(M)$.
Let ${\rm v}_k$ be the volume of the unit ball in ${\mathbb R}^k$. In Sect.\ref{mixedriemann}, we prove the following theorem.

\begin{theorem}\label {crofton mixed}
Let $\{Y, \Gamma, \mu\}= \{Y_1, \Gamma _1 ,\mu_1\}\times \ldots \times \{Y_m ,\Gamma _m ,\mu _m\}$
and let $n = n_1 +\ldots+n_m$. Then for any $n$-dimensional submanifold $M \subset Y= Y_1\times \ldots \times Y_m$
we have
$$\mathfrak M_{Y,\Gamma, \mu}(M) = C_1\cdot \ldots \cdot C_m\cdot {n!{\rm v}_n \over (n_1!{\rm v}_{n_1})\cdot
\ldots \cdot (n_m!{\rm v}_{n_m})}\cdot {\rm vol}_{g^*}(M),$$ 
where $g^* = (g_1^*, \ldots, g_1^*, \ldots, g_m^*,\ldots , g_m^*)$ is the mixing
$n$-tuple of non-negative quadratic forms $g_i^*$ coming from $g_i$ on $Y_i$,
in which each $g_i^*$ occurs $n_i$ times.
\end{theorem}

\noindent
In Sect.\ref{examples}, we apply Theorems \ref{crofton density}\,-\,\ref{crofton mixed}
for the derivation of two Crofton formulae. Using these, we consider some examples in which we compute the average number
of common zeros of $n$ smooth functions on a manifold of dimension $n$. 
It turns out that the answer
depends on the mixed volume of some Finsler ellipsoids.
We view these computations as a smooth analogue of the BKK theorem \cite{Be}.

\section{Preliminaries}\label{preliminaries}
We gather here some definitions and results from \cite{AK} which will be used later on. Let $V$ be a real vector space of dimension $n$.
Throughout the paper we use the following notations:

- ${\rm Gr}_m(V)$ is the Grassmanian, whose points are vector subspaces
of dimension $m$ in $V$;

 - ${\rm D}_m(V)\subset \bigwedge^m(V)$ is the cone of decomposable $m$-vectors;

- ${\rm AGr}^m(V)$ is the affine Grassmanian, whose points are affine
subspaces of codimension $m$ in $V$.

\medskip
\noindent
{\it The ring of normal measures.}
A signed Borel measure $\mu $ on ${\rm AGr}^m (V)$ is called {\it normal}, if $\mu $ is translation invariant and finite on compact sets. The space of normal measures on ${\rm AGr}^m(V)$ is denoted by ${\mathfrak m}_m(V)$. 
The graded vector space ${\mathfrak m}(V) = {\mathfrak m}_0(V) \oplus {\mathfrak m}_1(V) \oplus \ldots \oplus
{\mathfrak m}_n(V)$ has a structure of a graded ring. Namely,
let 
$${\mathcal D}_{p,q} = \{(G,H) \in {\rm AGr}^p(V) \times {\rm AGr}^q(V)\ \vert \ {\rm codim}\, G\cap H \ne p+q\}.$$
Then we have the mapping 
$$P_{p,q}: {\rm AGr}^p(V) \times {\rm AGr}^q(V) \setminus {\mathcal D}_{p,q} \to {\rm AGr}^{p+q}(V),$$
given
by $P_{p,q}(G,H) = G\cap H$. For $\mu \in {\mathfrak m}_p(V), \, \nu \in {\mathfrak m}_q(V), \, p+q \le n,$
the measure $\mu \cdot \nu \in{\mathfrak m}_{p+q}(V)$
is defined by
$(\mu \cdot \nu ) (D) = (\mu \times \nu)(P_{p,q}^{-1}(D)),$
where $D \subset {\rm AGr}^{p+q}(V)$ is a bounded domain.

\medskip
\noindent
{\it The ring of normal densities on a manifold.}
Recall that an $m$-density on $V$ is a continuous function 
$\delta : {\rm D}_m (V) \to {\mathbb R}$, such that $\delta (t\xi) = \vert t \vert \delta(\xi)$ for all $t\in {\mathbb R}$.
One example of an $m$-density is the absolute value of an exterior $m$-form.
An $m$-density on a manifold $X$ is an $m$-density $\delta _x$ on each tangent space $T_x$,
such that the assignment $x\mapsto \delta _x$ is continuous.
We consider a density on $V$ as a translation invariant density on the corresponding affine space.
Such a density is called {\it normal} if it is obtained from a normal measure
on ${\rm AGr}^m(V)$ by the construction which we now describe.  Namely, for $D \subset V$ put
$${\mathcal J}_{m,D} = \{H \in {\rm AGr}^m(V) \ \vert \ H \cap D \ne \emptyset \}$$
and  for $\xi_1,\ldots,\xi_m \in V$ denote by $\Pi_\xi \subset V$ the parallelotope generated by $\xi _i$.
Given $\mu \in {\mathfrak m}_m$,  define a normal $m$-density $\chi _m(\mu)$ by 
$$\chi _m(\mu)(\xi_1\wedge \ldots \wedge \xi_m ) = \mu ({\mathcal J}_{m,\Pi_\xi}).$$
Then we have the spaces of normal $m$-densities
${\mathfrak n}_m = {\mathfrak n}_m(V) , m = 0,1,\ldots,n$, their direct sum
${\mathfrak n}(V) =\oplus \, {\mathfrak n}_m(V)$ and the linear map
$\chi = \oplus \chi_m : {\mathfrak m}(V) \to {\mathfrak n}(V)$. It is easily seen that ${\rm Ker}\, \chi$
is a homogeneous ideal of the graded ring ${\mathfrak m}(V)$. Therefore ${\mathfrak n}(V)$ carries a structure of a graded ring, the {\it ring of normal densities on $V$}.
We refer to \cite{AK}, Sect.\,3.3., for the connection of ${\mathfrak n}(V)$ with the ring of valuations
on convex bodies defined by S.\,Alesker \cite{Al}.

\medskip
\noindent
 Let $X$ be a smooth manifold of dimension $n$. A normal $m$-density on $X$
is a continuous function $x \mapsto \delta _x \in {\mathfrak n}_m(T_xX)$.
The pointwise construction of product leads to the definition of {\it the ring of normal densities on $X$},
denoted by ${\mathfrak n}(X)$.
As an example of a normal density, we will prove in the next section
that the absolute value $\vert \omega \vert $ of a decomposable differential $m$-form $\omega $
is a normal $m$-density. Morever, one can show that for decomposable forms multiplication
of forms agrees with multiplicaion of densities. More precisely, if $\omega_1, \omega _2$ are decomposable forms then $\vert \omega _1 \wedge \omega  _2\vert = \vert \omega _1 \vert \cdot \vert \omega _2\vert$.

\medskip
\noindent
{\it Contravariance of ${\mathfrak n}(X)$}.
The assignment $X \to {\mathfrak n}(X)$ is a contravariant functor from the category of smooth manifolds to the
category of commutative graded rings. Namely, the pull-back of a normal density is normal
and the product of pull-backs is the pull-back of the product of normal densities. This property is based on the following
construction of {\it the pull-back of a normal measure}.  

For a linear map $\varphi: U \to V$ of vector spaces one can define the pull-back operation
$${\mathfrak m}_m (V) \ni \mu \mapsto \varphi^*(\mu)\in {\mathfrak m}_m(U),$$
so that $\chi _m (\varphi^*\mu ) = \varphi^*(\chi _m (\mu ))$.
For $\mu \in {\mathfrak m}_p(V), \nu \in {\mathfrak m}_q(V)$ one has
$$\chi _p(\varphi ^*\mu)\cdot \chi_q(\varphi ^*\nu) = \chi_{p+q}(\varphi ^*(\mu\cdot \nu)).$$
We need the notion of pull-back only for epimorphisms. If $\varphi $ is surjective then we have the closed 
imbedding 
$\varphi _*:{\rm AGr}^m(V) \to {\rm AGr}^m(U)$ defined by taking the preimage under $\varphi $ of an affine subspace in $V$. By definition, the measure $\varphi ^*\mu$ is
supported on $\varphi _*{\rm AGr}^m(V)$ and 
$$
\varphi^*\mu(\Omega) = \mu (\varphi _*^{-1}(\Omega ))$$ for any Borel set $\Omega \subset \varphi_*{\rm AGr}^m(V)$.

\medskip
\noindent
{\it Normal 1-densities.} It is shown in \cite{AK} that any smooth 1-density is normal.
This follows from certain properties of the cosine transform, see, e.g., \cite{AB}. Therefore ${\mathfrak n}(X)$ contains
the ring generated by smooth 1-densities.

\medskip
\noindent
{\it Normal densities and convex bodies.} Let $A_1\ldots, A_m$ be convex bodies in the dual space $V^*$. The $m$-density
$d_m(A_1, \ldots, A_m)$ on $V$ is defined as follows. 

\begin{definition}\label{d_m} Let $H$ be the subspace of $V$ generated by $\xi _1, \ldots, \xi_m\in V$,
$H^\bot \subset V^*$  the orthogonal complement to $H$, and $\pi _H: V^* \to V^*/H^\bot $
the projection map.
Consider $\xi_1\wedge \ldots \wedge \xi_m$ as a volume form on $V^*/H^\bot $.
Then
$d_m(A_1, \ldots, A_m)(\xi _1, \ldots, \xi _m)$ is the mixed
$m$-dimensonal volume of $\pi_HA_1, \ldots, \pi_HA_m$.

\end{definition}
\noindent
In particular,
$d_1(A)(\xi) = h(\xi) - h(-\xi )$, where $h : V \to { R}$ is the support
function of a convex body $A$. 
Furthermore, 
if $V$ and $V^*$ are identified with the help of some Euclidean metric then 

\begin{equation}\label{d_m(A)}d_m(A)(\xi_1, \ldots, \xi_m) = {\rm vol} _m
(\pi_HA )\cdot {\rm vol}_m (\Pi_\xi) \, , 
\end{equation}
where ${\rm vol}_m$ is the $m$-dimensional volume and $\Pi _\xi$ is the parallelotope generated by $\xi_1,\ldots, \xi_m$.

If $A$ is a smooth convex body of full dimension then the density $d_1(A)$
is smooth and, consequently, normal. Using the pull-back operaion
for normal measures and densities, one can show that 
$d_1(A)$ is normal for a smooth convex body of any dimension.
Suppose $A_1, \ldots, A_m$ are centrally symmetric smooth convex bodies. Then Theorem\,7 in \cite{AK}
asserts that
\begin{equation}
d_1(A_1)\cdot \ldots \cdot d_1(A_m) = m!\,d_m(A_1, \ldots, A_m). \label{*}
\end{equation}
Thus all densities $d_m(A_1, \ldots, A_m)$ are normal.

\medskip
\noindent
{\it Finsler convex sets.} Let $X$ be a smooth manifold of dimension $n$. Suppose that for every $x \in X$ we are given
a convex body ${\mathcal E}(x) \subset T_x^*$ depending continuously on $x\in X$.
The collection ${\mathcal E} = \{{\mathcal E}(x) \, \vert \, x \in X\}$ is called a Finsler convex set in $X$.

\medskip
\noindent
(1) Consider the domain
$$\bigcup _{x\in X}{\mathcal E }(x) \subset T^*(X).$$
By definition, the volume of ${\mathcal E}$ is the volume of this domain with respect to the standard symplectic structure on the cotangent bundle. Using Minkowski sum and homotheties, we 
can form linear combinations of convex sets with non-negative coefficients.
The linear combination of Finsler convex sets is defined by
$$(\sum _i \lambda _i {\mathcal E}_i)(x) = \sum _i \lambda _i {\mathcal E}_i(x).$$
For
$n$ Finsler convex sets ${\mathcal E}_1, \ldots, {\mathcal E}_n$ the volume of $\lambda _1 {\mathcal E}_1+\ldots+\lambda _n{\mathcal E}_n$ is a homogeneous polynomial of degree $n$ in $\lambda _1, \ldots, \lambda _n$. 
Its coefficient at $\lambda _1\cdot\ldots\cdot \lambda _n$ divided by $n!$ is called the mixed volume of Finsler convex sets ${\mathcal E}_1, \ldots, {\mathcal E}_n$.

\medskip
\noindent
(2) For
$m$ Finsler convex sets ${\mathcal E}_1, \ldots, {\mathcal E}_m$ denote by
$D_m({\mathcal E}_1, \ldots, {\mathcal E}_m)\in {\mathfrak n}_m(X)$ the $m$-density whose value at $x\in X$
equals $d_m({\mathcal E}_1(x), \ldots, {\mathcal E}_m(x))$. 
By definition, $D_m({\mathcal E}) = D_m({\mathcal E}_1, \ldots, {\mathcal E}_m)$,
where ${\mathcal E}_i = {\mathcal E}, i=1,\ldots,m$.
The integral
$\int _XD_n({\mathcal E}_1, \ldots, {\mathcal E}_n)$ is equal to the mixed volume of ${\mathcal E}_1, \ldots, {\mathcal E}_n$.

\medskip
\noindent
(3)
Under the assumptions of (\ref{*}) on $A_i = {\mathcal E}_i(x)$ at every $x \in X$ we have
\begin{equation}
D_1({\mathcal E}_1)\cdot \ldots \cdot D_1({\mathcal E}_m) = m!\, D_m({\mathcal E}_1, \ldots, {\mathcal E}_m). \label{**}
\end{equation}
Thus all densities $D_m({\mathcal E}_1, \ldots, {\mathcal E}_m)$ are normal.

\medskip
\noindent
(4) Given a $C^\infty $ map $f : X \to Y$ and a Finsler convex set ${\mathcal E}$ in $Y$, one can define its inverse image,
a Finsler convex set in $X$,  
by $f^*{\mathcal E}(x) = d^*f({\mathcal E}(f(x))$, where $d^*f$
is the dual to the differential of $f$. It is easy to see that 
$$D_m(f^*{\mathcal E}_1, \ldots, f^*{\mathcal E}_m) = f^*(D_m({\mathcal E}_1, \ldots, {\mathcal E}_m)), $$
i.e., the operation of inverse image of Finsler convex sets commutes with the pull-back of densities $D_m$.

\section{Decomposable forms and normal densities}\label{decomp}

\bigskip

In this section we prove Theorem\,\ref{crofton density}. The proof will be based on an auxiliary assertion  
(Prop.\,\ref{normal}). In order to state this, we start by recalling two important notions related to densities on smooth manifolds.

{\it Locally decomposable densities.} An exterior $m$-form $\omega $ on a manifold $X$ is said to be {\it locally decomposable}  
if $\omega $ is the product of 1-forms in a neighborhood of every point $x\in X$. 
Any top-order form is locally decomposable.
The pull-back of a locally decomposable form under a smooth mapping is also locally decomposable.
We will say that an $m$-density on $X$ is {\it locally decomposable} 
if
for every point $x \in X$
there is a neighborhood of $x$ in which $\delta $ is the absolute value of a decomposable $m$-form defined in that neighborhood. Any non-negative smooth top-order density is locally decomposable. The pull-back of a locally deomposable density is also locally decomposable. 

{\it Push-forward for densities.}  Let
$X,Y$ be smooth manifolds, $X$ a locally trivial fiber bundle over $Y$ with compact fiber,
$p: X \to Y$ the projection map, and $k$ the dimension of the fiber $F_y = p^{-1}(y),\ y \in Y$. Let $\delta $ be an $m$-density on $X$
and assume $k = {\rm dim}\,F_y \le m$. Then the push-forward $p_*\delta$ is an $(m-k)$-density  on $Y$,
defined as follows. Given $\eta _1, \ldots, \eta_{m-k} \in T_y(Y)$ and $x \in F_y$, choose arbitrarily $\xi_1, \ldots, \xi_{m-k} \in T_x(X)$,
so that $dp_x(\xi_i) = \eta _i$. Put $\xi = \xi_1 \wedge \ldots \wedge \xi_{m-k}$ and consider the 
density $\xi \lrcorner \,\delta $ on the fiber, defined by 
$$\xi\lrcorner \, \delta (f_1\wedge \ldots \wedge f_k) = \delta(\xi\wedge f_1 \ldots \wedge f_k),$$
where $f_1, \ldots, f_k$ are tangent to the fiber $F_y$ at $x$. Here, $\xi\lrcorner \, \delta (f_1\wedge \ldots \wedge f_k)$
does not depend on the choice of $\xi _i$. By definition,
$$p_*\delta (\eta _1 \wedge \ldots \wedge \eta_{m-k}) = \int _{F_{y}} \xi\lrcorner \, \delta.$$

\begin{proposition}\label{normal}
Let $\delta $ be a locally decomposable $m$-density on $X$. Then the push-forward $p _*\delta $ is a normal $(m-k)$-density on $Y$. 
\end{proposition}

\noindent
{\it Proof of Theorem\,\ref{crofton density}}.
The top-order density $\Phi $ is normal. Its pull-back $\delta = q^*\Phi $
is also normal. Moreover, $\delta $ is a locally decomposable form, and so $p_*\delta = p_*q^*\Phi $
is normal by Prop.\,\ref{normal}. The restriction of $p_*q^*\Phi $ to $M\subset Y$ is normal as the pull-back
of  normal density on $Y$.
\hfill $\square $

\medskip
\noindent
The proof of Prop.\,\ref{normal} requires several simple assertions from linear algebra. 
Let $V$  be a real vector space of dimension $n$ and $V^*$ the dual space.
For a decomposable non-zero $m$-form $\omega = \varphi _1 \wedge  \ldots  
\wedge \varphi_m$, where $\varphi _i \in V^*$, 
the subspace in $V^*$ generated by $\varphi _1, \ldots, \varphi_m$
has dimension $m$. 
Its orthogonal complement in $V$ is called the kernel of $\omega $ and is denoted by
${\rm Ker}\, \omega$. Of course, ${\rm Ker}\, \omega $ is given by the equations $\varphi _1 = \ldots =\varphi _m = 0 $.
A volume form is a non-zero form of maximal degree. A volume form on a given vector space
is unique up to a scalar multiple and decomposable.

\begin{lemma}\label {decomp1} An $m$-form $\omega $ is decomposable if and only if $\omega $ is the pull-back of a volume form $\omega^\prime $ under an epimorphism $\pi : V \to V^\prime$.

\end{lemma}

\begin{proof} Let $\omega $ be a decomposable $m$-form, $\omega = \varphi _1 \wedge \ldots \wedge \varphi _m$, 
and $V^\prime =V/{\rm Ker\, \omega} $. Denote by $\pi : V \to  V^\prime $ the canonical map. Then each $\varphi _i$ is the pull-back of
some $\varphi _i^\prime $, where $\varphi_1^\prime , \ldots,  
\varphi_m  ^\prime$ are linearly independent 1-forms on $V^\prime $. The wedge product $\omega ^\prime = \varphi _1^\prime \wedge
\ldots \wedge \varphi_m^\prime$ is a volume form on $V^\prime $ and its pull-back $\pi ^*\omega ^\prime $ is the initial form $\omega $.
The converse assertion is obvious.
\end{proof}

\noindent
For any $\omega \in \bigwedge ^m V^*$ and $\xi = (\xi _1, \ldots, \xi _k ),\ \xi_i \in V,\ k\le m$, we have the inner product 
$\xi \lrcorner \, \omega \in \bigwedge ^{m-k} V^*$, given by
$(\xi\lrcorner \, \omega) (v_1, \ldots, v_{m-k}) = \omega(\xi_1, \ldots, \xi_k, v_1, \ldots ,v_{m-k})$.

\begin{lemma}\label{decomp2} If $\omega \ne 0$ is decomposable then $\xi \lrcorner \, \omega $ is also 
decomposable. If $\xi _1, \ldots, \xi _k $ are linearly independent modulo ${\rm Ker}\, \omega$
then 
$${\rm Ker}\, \xi \lrcorner \, \omega = {\rm Ker}\, \omega +{\mathbb R}\, \xi _1 + \ldots + {\mathbb R}\, \xi_k. $$
Otherwise $\xi\lrcorner\, \omega = 0 $.
\end{lemma}

\begin{proof} Let $\omega = \pi^*\omega ^\prime$ be the pull-back of $\omega ^\prime \in \bigwedge ^m V^\prime$ under an epimorphism $\pi: V \to V^\prime$.
Write $\pi (\xi)$ for $(\pi (\xi_1),\ldots,\pi(\xi_k))$. 
Then $\xi\lrcorner \,\omega = \pi^*(\pi(\xi)\lrcorner\, \omega ^\prime)$ and 
${\rm Ker}\, \xi \lrcorner \, \omega  = \pi^{-1} \bigl ({\rm Ker}\, \pi(\xi) \lrcorner \, \omega ^\prime \bigr )$.
By Lemma \ref{decomp1} we may assume that  $\omega ^\prime $ is a volume form on $V^\prime$. 
If $\pi(\xi_1), \ldots, \pi (\xi_k)$ are
linearly independent then $\pi (\xi)\lrcorner \,\omega ^\prime$ is the pull-back of a volume form on $V^\prime /U $, where
$U = {\mathbb R}\pi (\xi_1) +\ldots + {\mathbb R}\pi (\xi_k).$  Thus $\xi \lrcorner\, \omega$ is decomposable with kernel  ${\rm Ker}\, \omega +{\mathbb R}\xi _1 + \ldots +{\mathbb R}\xi _k$. Otherwise 
$\pi (\xi) \lrcorner \, \omega ^\prime= 0 $ and $\xi \lrcorner \, \omega = 0 $.
\end{proof}

\noindent
Suppose $\omega \ne 0$ is a decomposable $m$-form on $V$, where $1\le m \le n $. We associate with $\omega  $ a
normal measure
$\mu_\omega \in {\mathfrak m}_m(V)$. Namely, by Lemma \ref{decomp1}  we have $\omega = \pi^*\omega ^\prime $, where
$\pi : V \to V^\prime = V/{\rm Ker}\, \omega$  is the canonical quotient map and
$\omega ^\prime $ is a volume form on $ V^\prime$.   Let $\mu ^\prime $ be the corresponding Lebesgue measure on $V^\prime $, so that
$\chi _m(\mu ^\prime ) = \vert \omega ^\prime \vert$.
Applying
the pull-back operation for normal measures, we put
 $\mu  _\omega= \pi ^*\mu ^\prime \in {\mathfrak m}_m(V)$. 
Then
$$\chi _m (\mu _\omega) = \chi _m (\pi ^*\mu ^\prime)  = \pi^*(\chi _m \mu^\prime) = \pi^*(\vert\omega ^\prime \vert) = 
\vert \omega \vert ,$$  
showing that $\vert \omega \vert$ is a normal $m$-density.
For $\omega = 0$ we put $\mu _0 = 0$.

\begin{lemma} \label{cont}The map
$${\rm D}_{m} (V^*) \ni \omega \mapsto \mu_\omega \in {\mathfrak m}_m$$
is continuous.
\end{lemma}
\begin{proof} Let $\omega _i \to \omega $ as $i \to \infty $. We want to prove that $\mu _{\omega _i} \to \mu_\omega $.
Put $L_i = {\rm Ker}\, \omega _i$. By compactness of ${\rm Gr}_{n-m}(V)$ we may assume that $L_i \to L$,
where $L$ is some vector subspace of codimension $m$ in $V$. Note that if $\omega \ne 0$ then $L = {\rm Ker}\, \omega $. 
Take any subspace $M \subset V$, such that $M \oplus L = V$. Then $M\oplus L_i = V$ for large $i$, so that we
have the projection maps $\pi $ and $ \pi _i$ of $V$ to $M$ parallel to $L$ and, respectively, $L_i$.
Consider the restrictions of $\omega $ and $\omega _i$ onto $M$.
These restrictions are top-order forms on $M$,
so let $\nu  $ and $\nu _i $ be the corresponding volume measures
on $M$.
Take any continuous function $f$ with compact support on ${\rm AGr}^m(V)$ and define the functions
$\tilde f_i, \tilde f$ on $M$ by
$$ \tilde f_i(m) = f(m+L_i), \tilde f(m) = f(m+L), \ m\in M.$$
Observe that $\tilde f_i$ and $\tilde f$ have compact supports bounded altogether and $\tilde f_i \to \tilde f$.
Since $\omega _ i \vert _M\to \omega \vert _M$, we have
 $\nu _i \to \nu $. Therefore
$$\int _{{\rm AGr}^m (V)}  f\cdot d\mu_{\omega_i} = \int _M \tilde f_i \cdot d\nu_i \to
\int _M \tilde f \cdot d\nu = \int _{{\rm AGr}^m (V)}  f\cdot d\mu _\omega $$
by the definiton of pull-back for normal measures applied to 
$\mu _\omega= \pi^*\nu $ and $\mu _{\omega_i}= \pi^*_i \nu_i$.
\end{proof}

\noindent
{\it Proof of Prop.\,\ref{normal}.}
Fix a point $x \in X$ and write $\delta = \vert \omega \vert$ in a neighborhood of $x$,
where $\omega $ is a decomposable $m$-form in that neighborhood. Put $V = T_x(X), W = T_y(Y)$, where $y = p(x)$, and $L = T_x(F_y)$.
Choose a complementary vector subspace $S$ to $L$ in $V$, so that $dp_x\vert _S $ is an isomorphism
 $S \buildrel \sim \over \to W$. By Lemma \ref{decomp2}, $f\lrcorner\, \omega $
 is a decomposable $(m-k)$-form on $V$ for any $k$-tuple $f=(f_1, \ldots, f_k)$, $f_i \in V$. Restrict this form to $S$ and then carry it down to $W$
using the isomorphism $S\simeq W$. For $f_1, \ldots, f_k \in L$ the resulting 
form does not depend on the choice of $S$. We call this form $\omega _x$. By construction, $\omega _x$
is a decomposable $(m-k)$-form on $W$ and
$$\omega _x (\eta _1, \ldots, \eta _{m-k}) = (f \lrcorner \, \omega )(\xi_1, \ldots, \xi_{m-k}),$$
where
$\xi_1, \ldots, \xi_{m-k} \in S$ are uniquely determined by $dp_x(\xi_i) =\eta_i$.
Let $\mu _x = \mu_{\omega _x} \in {\mathfrak m}_{m-k}(W)$ be the distinguished normal measure 
associated with $\omega _x$, such that $\chi_ {m-k}(\mu _x )= \vert\omega _x\vert = \delta _x$. When $x$ varies along the fiber $F_y$, the family
of forms $\{\omega _x\}$ is continuous. By Lemma \ref{cont} the family of measures $\{{\mu}_x\} $ is also
 continuous. Since $\mu _x$ and $\delta _x$ depend on $f_1\wedge \ldots \wedge f_k \in \bigwedge^kT_x(F_y)$ as densities, we can define a normal measure by the 
integral 
$
\mu = \int _{F_y}\mu_x$. Then
$$\chi _{m-k}(\mu ) (\eta _1 \wedge \ldots \wedge \eta _{m-k}) = \int _{F_y}\chi_{m-k}(\mu_x)(\eta _1\wedge \ldots
\wedge  \eta _{m-k})=$$
$$= \int _{F_y}\delta_x (\xi_1, \ldots, \xi_{m-k})=
\int _{F_y} \xi\lrcorner \, \delta _x . 
$$
By the definition of push-forward this equals $p_*\delta (\eta _1, \ldots, \eta _{m-k}).$
\hfill$\square $
\section{Product theorem}\label{prod}

For the proof of Theorem\,\ref {crofton product} below in this section we will need the following fact (Prop.\,\ref {product}).
Consider the product of two manifolds $X=X_1 \times X_2$
and use the same notation for differential forms on $X_1, X_2$ and their liftings to $ X$.
The same convention holds for densities.
Let $\delta _1, \delta _2$ be two
locally decomposable densities on $X_1$ and, respectively, $X_2$. As we know, $\delta _1, \delta _2$ are normal. 
Let $\delta _1 \in {\mathfrak n}_{m_1}(X_1), \delta _2\in {\mathfrak n}_{m_2}(X_2)$, $m= m_1 + m_2$, and
$\delta = \delta _1 \delta _2 \in {\mathfrak n}_m(X)$. Suppose $p_1:X_1 \to Y_1,\ p_2 : X_2 \to Y_2$ 
are two fiberings whose fibers are compact and have dimensions $k_1\le m_1$, $k_2 \le m_2$.
By Prop.\,\ref{normal} the push-forwards of $\delta _1, \delta _2$ and $\delta $
are normal densities.
\begin{proposition}\label{product}

One has $$(p_1\times p_2)_* (\delta) = p_{1*}(\delta _1)\cdot p_{2*}(\delta _2)$$
in the ring ${\mathfrak n}(Y_1\times Y_2)$.
\end{proposition}

\medskip
\noindent
{\it Proof of Theorem\,\ref{crofton product}.} For two Crofton data $\{Y_1,\Gamma_1,\mu_1\}$ and $\{Y_2, \Gamma _2, \mu _2\}$ let  $p_i: X_i \to Y_i , q_i: X_i \to \Gamma _i , i=1,2,$ be the corresponding locally trivial fiber bundles. The measures
$\mu _1, \mu_2$ are given by positive top-order smooth densities $\Phi _1 $ on $\Gamma _1$ and $
\Phi _2$ on $\Gamma _2$. Let $\Phi $ be their product on $\Gamma = \Gamma_1\times \Gamma _2$,
$\delta _i = q_i^*\Phi _i, i=1,2,$ and $\delta = (q_1\times q_2)^*(\Phi )$. 
Then 
$$\Omega _{Y_1\times Y_2, \Gamma _1 \times \Gamma _2, \mu_1 \times \mu_2}=
(p_1\times p_2)_*(\delta) = p_{1*}(\delta _1) \cdot p_{2*}(\delta _2)=$$
$$=
\pi_1^*\Omega _{Y_1,\Gamma _1, \mu_1}\cdot \pi_2^*\Omega _{Y_2,\Gamma_2,\mu_2}$$
by Prop.\,\ref{product}. The ring of normal densities depends contravariantly on a manifold, see 
Sect.\,\ref{preliminaries}. Thus the restriction to $M$ gives rise to a homomorphism of the rings of normal densities,
proving Theorem\,\ref{crofton product}.
\hfill $\square $

\medskip
\noindent
The product of normal measures $\mu $ on ${\rm AGr}^p(V)$ and $\nu $ on ${\rm AGr}^q(V)$ is a normal
measure $\mu \cdot \nu $ on ${\rm AGr}^{p+q}(V)$, see Sect.\ref{preliminaries}. 
We will apply this definition to $V= V_1 \oplus V_2$ and $\mu$, $\nu$
coming from $V_1$, $V_2$, respectively. More precisely, we 
consider the embeddings
$$\iota _1:{\rm AGr}^p(V_1) \to {\rm AGr}^p(V),\ \ G\mapsto G\oplus V_2,$$
$$\iota _2:{\rm AGr}^q(V_2) \to {\rm AGr}^q(V),\ \  H \mapsto V_1 \oplus H,$$
and identify ${\rm AGr}^p(V_1), {\rm AGr}^q(V_2)$ with their images under $\iota _1, \iota _2$.
For $\mu \in {\mathfrak m_p}(V_1)$ and $\nu \in {\mathfrak m_q}(V_2)$ we put 
$$\mu \cdot \nu = \iota _{1,*}(\mu )
\cdot \iota _{2,*}(\nu ) \in {\mathfrak m}_{p+q}(V).$$
Since $({\rm AGr}^p(V_1)\times{\rm AGr}^q(V_2))\cap{\mathcal D}_{p,q} = \emptyset $
and $P_{p,q} \cdot (\iota _1 \times \iota _2)$ is an embedding, we have $\mu \cdot \nu = \mu \times \nu $.

\medskip
\noindent
{\it Proof of Prop.\,\ref{product}}. Let $x_1 \in X_1, x_2 \in X_2$. As in the proof of 
Prop.\,\ref{normal}, construct the normal measures $\mu _{x_1}$ and $\mu _{x_2}$, such that
$\chi_{m_i-k_i}(\mu_{x_i}) = \delta _{x_i}, \ i=1,2.$ Then $\chi _{m-k}(\mu_{x_1}\cdot\mu_{x_2}) = \delta _{x_1}\cdot \delta
_{x_2} = \delta _x$. Now, $\mu _{x_1}$ and $\mu _{x_2}$ come from two different direct summands
of the decomposition $T_x(X) = T_{x_1}(X_1) \oplus T_{x_2}(X_2)$. Therefore, using the above observation, we
can replace the product of measures by the direct product, i.e.,
$\chi_{m -k}(\mu_{x_1}\times \mu _{x_2}) = \delta _x$.
Recall that $\mu _{x_1}$ and $\mu _{x_2}$ behave
as densities when $x_1$ and $x_2$ vary along the fibers $F_{y_1}=p_1^{-1}(y_1)$ and $F_{y_2} = 
p_1^{-1}(y_2)$, where $y_i 
=p_i(x_i)$.
We have
$$ \int _{F_{y_1}\times F_{y_2}} \mu_{x_1}\times\mu_{x_2}= \int _{F_{y_1}} \mu_{x_1} \times \int_{F_{y_2}}\mu_{x_2}.
$$
Take two decomposable covectors
$\eta _1 \in \bigwedge^{m_1-k_1}(T_{y_1}(Y_1))$ and
$\eta _2 \in \bigwedge^{m_2-k_2}(T_{y_2}(Y_2))$, apply $\chi $
and evaluate the both sides of the equality on $\eta_1\wedge\eta_2$ as in the proof of  Prop.\,\ref{normal}. This yields the result.
\hfill $\square $

\section{Mixed Riemannian volume}\label{mixedriemann}

In this section, we give the definition of the mixed Riemannian volume and prove Theorem\,\ref{crofton mixed}.
Let $g$ be a non-negative quadratic form on (the tangent bundle of) a manifold $X$.
Denote by $g_x$ the restriction of $g$ to $T_x(X)$.

\begin{definition}\label{vol}
For $\xi = \xi _1 \wedge \ldots \wedge \xi _k \in \bigwedge^kT_x(X)$ let ${\rm vol}_{k,g}(\xi)$ be
the $k$-dimensional $g$-volume of the parallelotope generated by $\xi_i$, i.e., the Euclidean volume of 
the parallelotope generated by the images of $\xi _i$ in $T_x(X)/{\rm Ker}\, g_x$.  
\end{definition} 

\noindent
The function $\sqrt {g_x}$
is convex and positively homogeneous of degree 1. Consider the convex body ${\mathcal T}_g(x)
\subset T_x^*(X)$ with support function $\sqrt {g_x}$. The Finsler convex body
${\mathcal T}_g= \{{\mathcal T}_g(x)\}$ consists of ellipsoids ${\mathcal T}_g(x)$ in the orthogonal complement
to the kernel of $g_x$.

\begin{lemma}\label{formula} $D_k({\mathcal T}_g) = {\rm v}_k\cdot {\rm vol}_{k,g}$, where ${\rm v}_k$
is the volume of the $k$-dimensional unit ball.
\end{lemma}
\begin{proof} Since the value of $D_k({\mathcal T}_g)$ at $x$ is $d_k({\mathcal T}_g(x))$, this follows from Definitions \ref{d_m} and \ref{vol}. In particular, if $g$ is non-degenerate then the assertion of the lemma is precisely the
equality (\ref{d_m(A)}). 
The general case is obtained by the continuity argument.
\end{proof}

\begin{corollary} The densities ${\rm vol}_{k,g}$ are normal.
\end{corollary}
\begin{proof}
This follows from the normality of  $D_k({\mathcal T}_g)$, see (\ref{**}).
\end{proof}

\begin{corollary} ${\rm vol}_{k,g} = {2^k\over k!{\rm v}_k}{\rm vol}^k_{1,g}$.\label{cor2}
\end{corollary}
\begin{proof}
We have
$${2^k\over k!{\rm v}_k}{\rm vol}^k_{1,g} = {1\over k!{\rm v}_k}D_1^k({\mathcal T}_g) = {1\over {\rm v}_k}D_k({\mathcal T}_g) = {\rm vol}_{k,g},$$
where the first equality follows from Lemma \ref {formula} (for $k=1$), the second from (\ref{**}), and the third again from Lemma \ref{formula}.
\end{proof}

\begin{definition}\label{mixedR} Let $g_1,\ldots,g_n$ be $n$ non-negative quadratic forms on an $n$-dimensional manifold $X$. 
The integral over $X$ of the $n$-density 
$${\rm vol}_{g_1, \ldots, g_n} = {2^n\over n!{\rm v}_n}{\rm vol}_{1,g_1}\cdot \ldots \cdot {\rm vol}_{1,g_n}$$
is called the mixed Riemannian volume of $X$ with respect to the mixing $n$-tuple $g_1,\ldots,g_n$.
\end{definition}
\noindent
The mixed Riemannian volume is denoted by ${\rm vol}_{g_1,\ldots,g_n}(X)$.
It follows from Corollary \ref{cor2} that for a Riemannian metric $g$ the mixed volume ${\rm vol}_{g,\ldots,g}(X)$
coincides with the Riemannian volume of $X$.
\begin{corollary} \label{finsler ellipsoids} ${\rm vol}_{g_1, \ldots, g_n}(X)$ is equal to the mixed volume of Finsler ellipsoids
${\mathcal T}_{g_1}, \ldots, {\mathcal T}_{g_n}$ divided by ${\rm v}_n$.
\end{corollary}
\begin{proof} Since $2\,{\rm vol}_{1,g_i} = D_1({\mathcal T}_{g_i})$, 
we have
$${\rm vol}_{g_1, \ldots, g_n} = 
{1\over n!{\rm v}_n}D_1({\mathcal T}_{g_1})\cdot \ldots \cdot D_1({\mathcal T}_{g_n})=
{1\over {\rm v}_n}D_n({\mathcal T}_{g_1}, \ldots, {\mathcal T}_{g_n}),$$
where the last equality follows from (\ref{**}).
\end{proof}
\medskip
\noindent
{\it Proof of Theorem\,\ref{crofton mixed}.} The Crofton densities on $Y_i$ are equal to
$$p_{i *}q_i^*\Phi _i = C_i\cdot{\rm vol}_{n_i, g_i} = C_i\cdot {2^{n_i}\over n_i!{\rm v}_{n_i}}\cdot {\rm vol}_{1,g_i}^{n_i}, $$ where $ i=1,\ldots, m$.
By Theorem\,\ref{crofton product} the Crofton density on $Y$ is their product, hence 

$${\mathfrak M}_{Y,\Gamma,\mu}(M)=
\int _M(p_{1*}q_1^*\Phi_1)\cdot \ldots \cdot (p_{m*}q_m^*\Phi_m)=$$
$$= 2^n\cdot \prod_{i=1}^m{C_i \over  n_i!{\rm v}_{n_i}}\cdot \int_M{\rm vol}_{1,g_1}^{n_1}\cdot \ldots \cdot {\rm vol}_{1,g_m}^{n_m} =$$
$$= n!{\rm v}_n\cdot \prod_{i=1}^m{C_i \over  n_i!{\rm v}_{n_i}}\cdot {\rm vol}_{g^*}(M)$$
by the definition of the mixed volume, where the mixing $n$-tuple $g^*$ is the one from the statement of Theorem\,\ref{crofton mixed}.
\hfill $\square $

\section{Examples}\label{examples}
Using our previous results,  
we obtain here Crofton formulae in two cases:
(1) for the product of Euclidean spaces;
(2) for the product of spheres.
For an arbitrary set of $n$ finite-dimensional Euclidean spaces
$V_1, \ldots, V_n \subset C^\infty (X)$ consider  
the system of equations
\begin{equation}f_1 - c_1 = \ldots = f_n - c_n = 0;\  0\ne f_i \in V_i, c_i \in{\mathbb R}. \label{1}
\end{equation}
Using the Crofton formula for the product of Euclidean spaces, we compute the average number of solutions over all such systems.
This number turns out to be equal to the mixed volume of certain Finsler ellipsoids or,
in other words, to the corresponding mixed Riemannian volume of $X$. For the system  
\begin{equation}
f_1=\ldots=f_n = 0 ;\  f_i \in V_i  \ , \label{2}
\end{equation}
a
similar result is contained in \cite{AK} and \cite{ZK}.
The computation for system (\ref{2}) in \cite{AK} uses the Crofton formula for the product of spheres.
The both computations can be regarded 
as smooth versions of the BKK theorem \cite{Be}.  

\medskip
\noindent{\it The product of Euclidean spaces.} Let $V$  be a Euclidean space, $g$ the quadratic form on $V$
defining the metric, $Y = V$,
and $\Gamma = {\rm AGr}^1(V)$. Take the volume density $\Phi $ on ${\rm AGr}^1(V)$
invariant
under rotations and translations. Normalize $\Phi $ by the condition that the corresponding measure $\mu $
equals 1 on the set of all hyperplanes intersecting a segment of length 1.
The classical Crofton formula for curves in a Euclidean space asserts that
$\Omega _{Y,\gamma,\mu} = {\rm vol}_{1,g}$.

\begin{corollary} \label{metric prod} Let $V_1, \ldots, V_n$ be Euclidean spaces, $g_1, \ldots, g_n$ their metric forms, $Y = V_1\times \ldots \times V_n$, $\Gamma = {\rm AGr}^1(V_1)\times \ldots \times {\rm AGr}^1(V_n)$,
and
$\mu = \mu _1 \times \ldots \times \mu_n$. Assume that all measures $\mu_i$ satisfy the above conditions.
Then
$$\Omega _{Y, \Gamma, \mu} = {\rm vol}_{1,h_1}\cdot \ldots \cdot {\rm vol}_{1,h_n},$$
where $h_1, \ldots, h_n$ are the quadratic forms on $Y = V_1\times \ldots \times V_n$ given by $h_i(v_1, \ldots, v_n) = g_i(v_i), i=1,\ldots,n$.
\end{corollary}
\begin{proof} This follows from Theorem \ref{crofton product}.
\end{proof}
\noindent
Let $X$ be a manifold of dimension $n$ and let
$V_1, \ldots, V_n$ be finite dimensional vector spaces of real $C^\infty $ functions
on $X$. 
Assume that each $V_i$ is equipped with a scalar product. 
To define the average number of solutions over all systems (\ref{1}), denote
by $V_i^*$ the dual Euclidean space and regard $f_i$ as a linear
functional $f_i:V_i^* \to {\mathbb R}$. System (\ref{1}) can be viewed as a set of hyperplanes $H_i \subset V_i^*$
having equations $f_i = a_i$. The number of solutions is denoted by $N(H_1, \ldots, H_n)$.
The average number of solutions is, by definition, the integral
$$\int _{{\rm AGr}^1(V_1^*)\times \ldots \times{\rm AGr}^1(V_n^*)}\ N(H_1, \ldots ,H_n)\cdot d\mu.$$
Let $\theta _i : X \to V_i^*$ be the map assigning to any $x \in X$
the functional
$$\theta _i(x)(f) = f(x),\ \ f\in V_i.$$

\begin{proposition}\label{average1}  Define the non-negative quadratic forms $h_i$ on $X$ by $h_i
=\theta _i^*g_i$, where $g_i$ are the metric quadratic forms on $V_i^*$.
Then the average number of solutions of systems (\ref {1})
is equal to
$$\int_X{\rm vol}_{1,h_1}\cdot \ldots \cdot {\rm vol}_{1,h_n}.$$
\end{proposition}
\begin{proof}
If $\Theta = \theta _1 \times \ldots \times \theta _n$ is an ebedding then the asserion follows from Cor.\,\ref{metric prod}
by the contravariance argument. If $\Theta $ is not an embeding then the proof
is obtined by application of Sard's lemma as it is done in \cite{AK}, Sect.\,4.4. 
\end{proof}

\begin{proposition}\label{average2}
Consider the Finsler convex bodies ${\mathcal T}_i$ in $X$, such that each ${\mathcal T}_i(x)$ is a convex body
in $T_x^*(X)$ with support function $\sqrt {(h_i)_x}$,
where $h_i$ are defined in Prop.\,\ref{average1}. Then the average number of solutions of systems (\ref {1})
is equal to

{\rm (i)} the mixed Riemannian volume of $X$ multiplied by ${n!{\rm v}_n\over 2^n}$ with respect to the mixing
$n$-tuple $h_1, \ldots, h_n$;

{\rm (ii)} the mixed volume of ${\mathcal T}_1, \ldots, {\mathcal T}_n$ 
multiplied by $n!\over 2^n$.
\end{proposition}

\begin{proof} Assertion (i) follows from Prop.\,\ref{average1} and  from Def.\,\ref{mixedR}.
Assertion (ii) follows from the equality $2{\rm vol}_{1,h_i} = D_1({\mathcal T}_i) $ (see Lemma \ref{formula}) and from
(\ref{**}). \end{proof}

\medskip
\noindent
{\it The product of spheres.} The Crofton formula for one sphere is classical, see \cite{Sa}.
The case of several spheres is considered in \cite{AK}.
We want to show here that the Crofton formula for the product of spheres is a consequence of
our previous results in the present paper. Let $V_1, \ldots, V_n$ be Euclidean vector spaces with metric forms $g_i$,
${\rm dim}\, V_i = n_i + 1$. Denote by $S_i$ the unit sphere in $V_i$, $Y_i = S_i, \Gamma _i
={\rm Gr}_{n_i}(V_i)$, $\mu _i$ the normalized orthogonally invariant measure on $\Gamma _i$, and $Y_\gamma = S_i\cap\gamma$ for $\gamma \in \Gamma _i$. 

\begin{corollary} Let $Y = S_1\times \ldots \times S_n, \Gamma = \Gamma _1 \times \ldots \times \Gamma _n$,
and $\mu = \mu_1 \times \ldots \times \mu_n$. Then:

{\rm (i)} $\Omega _{Y,\Gamma,\mu} = {1\over \pi ^n}{\rm vol }_{1,h_1}\cdot \ldots \cdot {\rm vol}_{1,h_n}$,
where the quadratic form $h_i$ on $Y =V_1 \times \ldots \times V_n$
is given by $h_i(v_1, \ldots,v_n) = g_i(v_i)$;

{\rm (ii)} for any submanifold $M \subset Y $ of dimension $n$ 
we have
$${\mathfrak M}_{Y,\Gamma,\mu}(M) = {n!v_n\over(2\pi)^n}\cdot {\rm vol}_{h_1, \ldots, h_n}(M).$$
\end{corollary}
\begin{proof} The Crofton formula for curves in a sphere tells us that $\Omega _{Y_i, \Gamma _i, \mu_i} = {1\over \pi}{\rm vol}_{1,g_i}$. Therefore it suffices to apply Theorem \ref{crofton product} and Def.\,\ref{mixedR}.
\end{proof}

\end {document}